\documentclass[letterpaper, 10 pt, conference]{ieeeconf}

\def\bsg{{\boldsymbol{g}}}

\def\bsv{{\boldsymbol{v}}}

\def\bsx{{\boldsymbol{x}}}

\def\bsH{{\boldsymbol{H}}}

\def\bsK{{\boldsymbol{K}}}
\def\bsL{{\boldsymbol{L}}}

\def\bsM{{\boldsymbol{M}}}

\def\bsQ{{\boldsymbol{Q}}}

\usepackage{geometry}
\usepackage{psfrag,cite,amsmath,epsfig,latexsym,color}
\usepackage{graphicx}
\graphicspath{{./figures/}}

\usepackage{amsthm}

\usepackage[scriptsize,bf]{caption}
\usepackage{color,xcolor,multirow}
\usepackage{amssymb}
\usepackage{bbm}
\usepackage{setspace}
\usepackage{array}
\usepackage{booktabs}
\usepackage{hyperref}
\usepackage{bm}
\usepackage{adjustbox}
\usepackage{makecell}

\usepackage[T1]{fontenc}
\usepackage{pifont}

\usepackage{algorithm,algorithmic}

\newtheorem{theorem}{Theorem}
\newtheorem{assumption}{Assumption}

\newtheorem{lemma}{Lemma}

\newtheorem{definition}{Definition}
\newtheorem{remark}{Remark}

\def\sgn{\mathop{\rm sgn}}

\newcommand{\diag}{{\rm diag}}

\DeclareMathOperator{\col}{col}

\DeclareMathOperator{\Deg}{Deg}
\allowdisplaybreaks

\newcommand{\red}[1]{{\color{red} #1}}

\begin{document}

\IEEEoverridecommandlockouts
\title{\bf Accelerated Zeroth-order Algorithm for Stochastic Distributed Nonconvex Optimization}


\author{Shengjun~Zhang, Colleen~P.~Bailey \thanks{
		Shengjun~Zhang and Colleen~P.~Bailey are with OSCAR Laboratory, Department of Electrical Engineering, University of North Texas, Denton, TX 76207 USA, email: {\tt\small ShengjunZhang@my.unt.edu}, {\tt\small Colleen.Bailey@unt.edu}.}}

\maketitle
\begin{abstract}
This paper investigates how to accelerate the convergence of distributed optimization algorithms on nonconvex problems with zeroth-order information available only.
We propose a zeroth-order~(ZO) distributed primal-dual stochastic coordinates algorithm equipped with ``powerball'' method to accelerate.
We prove that the proposed algorithm has a convergence rate of $\mathcal{O}(\sqrt{p}/\sqrt{nT})$ for general nonconvex cost functions.
We consider solving the generation of adversarial examples from black-box DNNs problem to compare with the existing state-of-the-art centralized and distributed ZO algorithms.
The numerical results demonstrate the faster convergence rate of the proposed algorithm and match the theoretical analysis.

\end{abstract}


\section{Introduction\label{zo_pb:sec:Introduction}}
In this paper, we focus on solving the following stochastic distributed nonconvex optimization problems
\begin{align}\label{zo_pb:eqn:xopt}
 \min_{x\in \mathbb{R}^p} f(x)=\frac{1}{n}\sum_{i=1}^n\mathbb{E}_{\xi_i}[F_i(x,\xi_i)],
\end{align}
where $n$ is the total number of agents, $x$ is the decision variable, $\xi_i$ is a random variable with dimension $p$, and $F_i(\cdot,\xi_i): \mathbb{R}^{p}\rightarrow \mathbb{R}$ is the stochastic function.
Various algorithms utilizing gradient information aiming to solve such problems in the form of~\eqref{zo_pb:eqn:xopt} have been proposed and applied to many applications.
However, in many realistic problems, it is unable or too expensive to achieve the gradient information \cite{conn2009introduction,audet2017derivative,larson2019derivative}.
For example, the simulation based optimization problems \cite{spall2005introduction}, the black-box universal attacking of deep neural networks problems \cite{goodfellow2014explaining, liu2018zeroth, chen2019zo}, just to name a few.
Because of the unavailability of gradient information, we consider that agent $i$ is only able to access its own stochastic ZO information $F_i(x,\xi_i)$.
For each agent $i$, the local cost function $f_i(x)$ is the expectation of the ZO information $\mathbb{E}_{\xi_i}[F_i(x,\xi_i)]$.
Agents communicate with their neighbors via an undirected communication network graph $\mathcal G$.

In recent years, distributed ZO optimization problems obtained more and more attention and have been applied into networked \cite{yang2019survey}.
\cite{yuan2014randomized,sahu2018distributed,wang2019distributed,
pang2019randomized} consider distributed gradient descent methods with ZO information.
\cite{yuan2015gradient} focus on applying the push-sum technique in distributed ZO optimization in order to handle direct communication between agents.
\cite{yu2019distributed} provided distributed ZO mirror descent algorithm.
\cite{tang2020distributedzero} utilized the gradient tracking technique in distributed ZO optimization problems.
\cite{beznosikov2019derivative} proposed distributed ZO sliding algorithms.
\cite{hajinezhad2019zone,yi2019linear,yi2021zerothorder,zhang2021convergence} combined primal--dual techniques and ZO information.

However, under the stochastic distributed settings in the exact form of~\eqref{zo_pb:eqn:xopt}, only a few works \cite{hajinezhad2019zone,yi2021zerothorder,zhang2021convergence} exist in the literature. 
ZONE-M in \cite{hajinezhad2019zone} achieves the convergence rate of $\mathcal{O}(p^2n/T)$ with a a very high sampling size of $\mathcal{O}(T)$ per iteration.
ZODPDA in \cite{yi2021zerothorder} achieves the convergence rate of $\mathcal{O}(\sqrt{p}/\sqrt{nT})$, which is the best known convergence rate so far.
ZODIAC in \cite{zhang2021convergence} has the convergence rate of $\mathcal{O}(\sqrt{p}/\sqrt{T})$.
Both ZODPDA and ZODIAC have $2n$ points sampled per iteration, which are more suitable for high dimensional decision variables cases.

The contributions of this work are summarized in the following:
\subsubsection{}
We propose an accelerated ZO algorithm based on primal--dual framework for stochastic distributed nonconvex optimization problems and prove the convergence rate of $\mathcal{O}(\sqrt{p}/\sqrt{nT})$.
To our best knowledge, the proposed algorithm is the first accelerated method in distributed ZO optimization literature.
\subsubsection{}
ZODIAC in \cite{zhang2021convergence} can be considered as a special case in the proposed algorithm.
We theoretically improve the convergence rate of ZODIAC from $\mathcal{O}(\sqrt{p}/\sqrt{T})$ to $\mathcal{O}(\sqrt{p}/\sqrt{nT})$.
\subsubsection{}
Extensive numerical examples are provided to demostrate the efficacy of the considered algorithm through benchmark examples on a large-scale agents systems.
The rest of this paper is organized as follows.
Section~\ref{zo_pb:sec-preliminary} introduces some preliminary concepts. Sections~\ref{zo_pb:sec-alg} introduces the proposed algorithm and analyzes its convergence properties. Simulations are presented in Section~\ref{zo_pb:sec:exp}. Finally, concluding remarks are offered in Section~\ref{zo_pb:sec:conclusions}.

\noindent {\bf Notations}: $\mathbb{N}_0$ and $\mathbb{N}_+$ denote the set of nonnegative and positive integers, respectively. $[n]$ denotes the set $\{1,\dots,n\}$ for any $n\in\mathbb{N}_+$.
$\|\cdot\|$ represents the Euclidean norm for vectors or the induced 2-norm for matrices. $\mathbb{B}^p$ and $\mathbb{S}^p$ are the unit ball and sphere centered around the origin in $\mathbb{R}^p$ under Euclidean norm, respectively. Given a differentiable function $f$, $\nabla f$ denotes the gradient of $f$.
${\bm 1}_n$ (${\bm 0}_n$) denotes the column one (zero) vector of dimension $n$. $\col(z_1,\dots,z_k)$ is the concatenated column vector of vectors $z_i\in\mathbb{R}^{p_i},~i\in[k]$. ${\bm I}_n$ is the $n$-dimensional identity matrix. Given a vector $[x_1,\dots,x_n]^\top\in\mathbb{R}^n$, $\diag([x_1,\dots,x_n])$ is a diagonal matrix with the $i$-th diagonal element being $x_i$.  The notation $A\otimes B$ denotes the Kronecker product
of matrices $A$ and $B$. Moreover, we denote $\bsx=\col(x_1,\dots,x_n)$, $\bar{x}=\frac{1}{n}({\bm 1}_n^\top\otimes{\bm I}_p)\bsx$, $\bar{\bsx}={\bm 1}_n\otimes\bar{x}$.
$\rho(\cdot)$ stands for the spectral radius for matrices and $\rho_2(\cdot)$ indicates the minimum
positive eigenvalue for matrices having positive eigenvalues.

\section{Preliminaries}\label{zo_pb:sec-preliminary}
The following section discusses some background on graph theory, smooth functions, the gradient
estimator, and additional assumptions used in this paper.

\subsection{Graph Theory}
Agents communicate with their neighbors through an underlying network, which is modeled by an undirected graph $\mathcal G=(\mathcal V,\mathcal E)$, where $\mathcal V =\{1,\dots,n\}$ is the agent set, $\mathcal E
\subseteq \mathcal V \times \mathcal V$ is the edge set, and $(i,j)\in \mathcal E$ if agents $i$ and $j$ can communicate with each other.
For an undirected graph $\mathcal G=(\mathcal V,\mathcal E)$, let $\mathcal{A}=(a_{ij})$ be the associated weighted adjacency matrix with $a_{ij}>0$ if $(i,j)\in \mathcal E$ if $a_{ij}>0$ and zero otherwise. It is assumed that $a_{ii}=0$ for all $i\in [n]$. Let $\deg_i=\sum\limits_{j=1}^{n}a_{ij}$ denotes the weighted degree of vertex $i$. The degree matrix of graph $\mathcal G$ is $\Deg=\diag([\deg_1, \cdots, \deg_n])$. The Laplacian matrix is $L=(L_{ij})=\Deg-\mathcal{A}$. Additionally, we denote $K_n={\bm I}_n-\frac{1}{n}{\bm 1}_n{\bm 1}^{\top}_n$, $\bsL=L\otimes {\bm I}_p$, $\bsK=K_n\otimes {\bm I}_p$, $\bsH=\frac{1}{n}({\bm 1}_n{\bm 1}_n^\top\otimes{\bm I}_p)$. Moreover, from Lemmas~1 and 2 in \cite{Yi2018distributed}, we know there exists an orthogonal matrix $[r \ R]\in \mathbb{R}^{n \times n}$ with $r=\frac{1}{\sqrt{n}}\mathbf{1}_n$ and $R \in \mathbb{R}^{n\times (n-1)}$ such that $R\Lambda_1^{-1}R^{\top}L=LR\Lambda_1^{-1}R^{\top}=K_n$, and $\frac{1}{\rho(L)}K_n\leq R\Lambda_1^{-1}R^{\top}\le\frac{1}{\rho_2(L)}K_n$, where $\Lambda_1=\diag([\lambda_2,\dots,\lambda_n])$ with $0<\lambda_2\leq\dots\leq\lambda_n$ being the eigenvalues of the Laplacian matrix $L$.

\subsection{Smooth Function}
\begin{definition}
A function $f(x):~\mathbb{R}^p\mapsto\mathbb{R}$ is smooth with constant $L_f>0$ if it is differentiable and
\begin{align}\label{zo_pb:nonconvex:smooth}
\|\nabla f(x)-\nabla f(y)\|\le L_{f}\|x-y\|,~\forall x,y\in \mathbb{R}^p.
\end{align}
\end{definition}

\subsection{Gradient Approximation}

Denote a random subset of the coordinates $\mathcal{S} \subseteq \{1, 2, \dots, p\}$ where the cardinality of $\mathcal{S}$ is $|\mathcal{S}| = n_{c}$. 
We provide two options of gradient approximation, denoted $g^e_{i}$ and defined by~\eqref{zo_pb:gradient:model2-st} and~\eqref{zo_pb:gradient:model2-st2}.

\begin{align}
&g^e_{i}=\frac{p}{n_{c}}\sum_{i\in \mathcal{S}}\frac{(F(x+\delta_{i} e_{i},\xi)-F(x,\xi))}{\delta_{i}}e_{i}
\label{zo_pb:gradient:model2-st}
\end{align}

\begin{align}
&g^e_{i}=\frac{p}{n_{c}}\sum_{i\in \mathcal{S}}\frac{(F(x+\delta_{i} e_{i},\xi)-F(x-\delta_{i} e_{i},\xi))}{2\delta_{i}}e_{i}
\label{zo_pb:gradient:model2-st2}
\end{align}

The coordinates are sampled uniformly, i.e. $\text{Pr}(i\in \mathcal{S}) = n_{c}/p$, which guarantees that both~\eqref{zo_pb:gradient:model2-st} and~\eqref{zo_pb:gradient:model2-st2} are \textit{unbiased} estimators of the \textit{full} coordinate gradient estimator $\sum_{i=1}^{d}\frac{(F(x+\delta_{i} e_{i},\xi)-F(x-\delta_{i} e_{i},\xi))}{2\delta_{i}}e_{i}$\cite{sharma2020zeroth}.

\subsection{Powerball Term}

Define the function 
\begin{equation}\label{zo_pb:pb:def}
\sigma(x, \gamma) = \sgn(x) |x|^{\gamma}
\end{equation}
where $\gamma \in [\frac{1}{2}, 1]$.
Note that when $\gamma = 1$, $\sigma(x, 1)$ reduces to $x$.
Unlike the ``powerball'' terms in \cite{zhou2020pbsgd} and \cite{yuan2019powerball}, under distributed settings, the range of $\gamma$ has to be modified \cite{zhang2021accelerated}.
\subsection{Assumptions}

\begin{assumption}\label{zo_pb:ass:graph}
The undirected graph $\mathcal G$ is connected.
\end{assumption}

\begin{assumption}\label{zo_pb:ass:optset}
The optimal set $\mathbb{X}^*$ is nonempty and the optimal value $f^*>-\infty$.
\end{assumption}

\begin{assumption}\label{zo_pb:ass:zeroth-smooth}
For almost all $\xi_i$, the stochastic ZO oracle $F_i(\cdot,\xi_i)$ is smooth with constant $L_f>0$.
\end{assumption}

\begin{assumption}\label{zo_pb:ass:zeroth-variance}
The stochastic gradient $\nabla_xF_i(x,\xi_i)$ has bounded variance for any $j$th coordinate of $x$, i.e., there exists $\zeta\in\mathbb{R}$ such that $\mathbb{E}_{\xi_i}[(\nabla_xF_i(x,\xi_i)-\nabla f_i(x))_{j}^2]\le\zeta^2,~\forall i\in[n],~\forall j\in[p],~\forall x\in\mathbb{R}^p$. It also implies that $\mathbb{E}_{\xi_i}[\|\nabla_xF_i(x,\xi_i)-\nabla f_i(x)\|^2]\le\sigma^2_1\triangleq p \zeta^2,~\forall i\in[n],~\forall x\in\mathbb{R}^p$.
\end{assumption}

\begin{assumption}\label{zo_pb:ass:fig}
Local cost functions are similar, i.e.,
there exists $\sigma_2\in\mathbb{R}$ such that $\|\nabla f_i(x)-\nabla f(x)\|^2\le\sigma^2_2,~\forall i\in[n],~\forall x\in\mathbb{R}^p$.
\end{assumption}


\section{Algorithm}\label{zo_pb:sec-alg}

\subsection{Algorithm Description}
We consider the novel distributed primal-dual framework in~\cite{zhang2021convergence} and apply the ``powerball'' term described in~\eqref{zo_pb:pb:def} directly on the estimations of gradient.
We summarize the proposed method ZODIAC-PB as Algorithm~\ref{zo_pb:algorithm-random-pd}.

\begin{algorithm}[!ht]
\caption{ZODIAC-PB}
\label{zo_pb:algorithm-random-pd}
\begin{algorithmic}[1]
\STATE \textbf{Input}: positive number $\alpha$, $\beta$, $\eta$, and positive sequences $\{\delta_{i,k}\}$, $\gamma$.
\STATE \textbf{Initialize}: $ x_{i,0}\in\mathbb{R}^p$ and $v_{i,0}={\bm 0}_p,~
\forall i\in[n]$.
\FOR{$k=0,1,\dots$}
\FOR{$i=1,\dots,n$  in parallel}
\STATE  Broadcast $x_{i,k}$ to $\mathcal{N}_i$ and receive $x_{j,k}$ from $j\in\mathcal{N}_i$;
\STATE Select coordinates independently and uniformly;
\STATE Select $\xi_{i,k}$ independently;\\
\STATE \textbf{Option 1:} sample $F_i(x_{i,k}+\delta_{i,k}e_{i,k},\xi_{i,k})$, and $F_i(x_{i,k},\xi_{i,k})$;
\STATE  Update $x_{i,k+1}$ by \eqref{zo_pb:alg:random-pd-x} with~\eqref{zo_pb:gradient:model2-st};
\STATE \textbf{Option 2:} sample  $F_i(x_{i,k}+\delta_{i,k}e_{i,k},\xi_{i,k})$ and $F_i(x_{i,k}-\delta_{i,k}e_{i,k},\xi_{i,k})$;
\STATE  Update $x_{i,k+1}$ by \eqref{zo_pb:alg:random-pd-x} with~\eqref{zo_pb:gradient:model2-st2};
\STATE  Update $v_{i,k+1}$ by \eqref{zo_pb:alg:random-pd-q}.
\ENDFOR
\ENDFOR
\STATE  \textbf{Output}: $\{\bsx_{k}\}$.
\end{algorithmic}
\end{algorithm}

\begin{subequations}\label{zo_pb:alg:random-pd}
\begin{align}
x_{i,k+1} &= x_{i,k}-\eta\Big(\alpha\sum_{j=1}^nL_{ij}x_{j,k}+\beta v_{i,k}+\sigma(g^e_{i,k}, \gamma)\Big), \label{zo_pb:alg:random-pd-x}\\
v_{i,k+1} &=v_{i,k}+ \eta\beta\sum_{j=1}^n L_{ij}x_{j,k}, \nonumber \\
& \forall x_{i,0}\in\mathbb{R}^p, \sum_{j=1}^nv_{j,0}={\bm 0}_p,~\forall i\in[n]. \label{zo_pb:alg:random-pd-q}
\end{align}
\end{subequations}

\subsection{Convergence Analysis}

\begin{theorem}\label{zo_pb:thm-sg-smT}
Suppose Assumptions~\ref{zo_pb:ass:graph}--\ref{zo_pb:ass:fig} hold. For any given $T>n^3/p$, let $\{\bsx_k,k=0,\dots,T\}$ be the output generated by Algorithm~\ref{zo_pb:algorithm-random-pd} with
\begin{align}\label{zo_pb:step:eta2-sm}
&\alpha=\kappa_1\beta,~\beta=\frac{\kappa_2\sqrt{pT}}{\sqrt{n}},~ \eta=\frac{\kappa_2}{\beta},\nonumber\\
&\delta_{i,k}\le\frac{\kappa_\delta}{p^{\frac{1}{4}}n^{\frac{1}{4}}(k+1)^{\frac{1}{4}}},~\forall k\le T,
\end{align}
where $\kappa_1>\frac{1}{\rho_2(L)}+1$, $\kappa_2\in\Big(0,\min\{\frac{(\kappa_1-1)\rho_2(L)-1}{\rho(L)+(2\kappa_1^2+1)\rho(L^2)+1},\frac{1}{5}\}\Big)$,  and $\kappa_\delta>0$, then we have,
\begin{subequations}
\begin{align}
&\frac{1}{T}\sum_{k=0}^{T-1}\mathbb{E}[\|\nabla f(\bar{x}_k)\|_{1+\gamma}^2]
=\mathcal{O}(\frac{\sqrt{p}}{\sqrt{nT}})
+\mathcal{O}(\frac{n}{T}),\label{zo_pb:coro-sg-sm-equ3}\\
&\mathbb{E}[f(\bar{x}_{T})]-f^*=\mathcal{O}(1),\label{zo_pb:coro-sg-sm-equ4}\\
&\frac{1}{T}\sum_{k=0}^{T-1}\mathbb{E}\Big[\frac{1}{n}\sum_{i=1}^{n}\|x_{i,k}-\bar{x}_k\|^2\Big]
=\mathcal{O}(\frac{n}{T}).\label{zo_pb:coro-sg-sm-equ3.1}
\end{align}
\end{subequations}
\end{theorem}

In order to prove Theorem~\ref{zo_pb:thm-sg-smT}, we introduce the following lemmas.

\begin{lemma}(Lemma 2 in~\cite{zhang2021convergence})\label{zo_pb:lemma:variance}
Consider $f(x) = \mathbb{E}_{\xi} [F(x, \xi)]$, we have the following relationship,
\begin{align}
&\mathbb{E} \Big[ \| g^e_{i} \|  ^2\Big] \nonumber \\
&\leq 2(p-1)\left\| \nabla f (x) \right\|^2 + 2p \sigma^2_1 + \frac{3p^{2}}{n_{c}} \left( \zeta^2 + \frac{L_{f}^2 \delta_{k}^2}{2} \right) \nonumber \\
&\quad+ \frac{p^{2}L_{f}^2 \delta_{k}^2}{2} \label{zo_pb:eq_bd_var_CGE}
\end{align}
\end{lemma}
where $\delta_{k} = \max\{ \delta_{i} \}, i \in [p]$.

\begin{lemma}\label{zo_pb:lemma:pb}
By using the powerball term in~\eqref{zo_pb:pb:def} and when $\gamma \in [\frac{1}{2}, 1)$, we have $\Big\Vert\sigma(g^e_{i}, \gamma)\Big\Vert^2 \leq \Big\Vert g^e_{i} \Big\Vert^{2}_{1+\gamma}$.
\end{lemma}
\begin{proof}
The proof follows the proof of Lemma 1 in~\cite{zhang2021accelerated} directly.
\end{proof}

\begin{lemma}\label{zo_pb:zerosg:lemma:grad-st}
Suppose Assumptions~~\ref{zo_pb:ass:zeroth-smooth}--~\ref{zo_pb:ass:fig} hold. Let $\{\bsx_k\}$ be the sequence generated by Algorithm~\ref{zo_pb:algorithm-random-pd}, $\bsg^e_k=\col(g^e_{1,k},\dots,g^e_{n,k})$, $\bsg^0_k=n\nabla{f}(\bar{\bsx}_k)$, $\bar{\bsg}_k^0=\bsH\bsg^0_{k}={\bm 1}_n\otimes\nabla f(\bar{x}_k)$, then
\begin{subequations}
\begin{align}
\mathbb{E}\Big[\|\bsg^e_k\|_{1+\gamma}^2\Big]
&\le  6(p-1)\|\bar{\bsg}_{k}^0\|_{1+\gamma}^2+6(p-1)L_f^2\|\bsx_{k}\|^2_{\bsK}\nonumber\\
&\quad+6n(p-1)\sigma^2_2 + \frac{3np^{2}}{n_{c}} \left( \zeta^2 + \frac{L_{f}^2 \delta_{k}^2}{2} \right)\nonumber \\
&\quad+2np \sigma^2_1+ \frac{np^{2}L_{f}^2 \delta_{k}^2}{2}\label{zo_pb:zerosg:rand-grad-esti2}\\
\|\bsg^0_{k+1}\|^2&\le 3(\eta^2L_f^2\|\bsg^e_{k}\|^2+n\sigma^2_2
+\|\bar{\bsg}_{k}^0\|^2).\label{zo_pb:zerosg:rand-grad-esti4}
\end{align}
\end{subequations}
\end{lemma}
\begin{proof}
(i) Eq.~\eqref{zo_pb:zerosg:rand-grad-esti2} is due to Lemma~\ref{zo_pb:lemma:variance}, Lemma~\ref{zo_pb:lemma:pb}, Cauchy-Schwarz inequality and Assumption~\ref{zo_pb:ass:fig}.

(ii) Eq.~\eqref{zo_pb:zerosg:rand-grad-esti4} is established by Cauchy-Schwarz inequality, Assumption~\ref{zo_pb:ass:zeroth-smooth} and~\ref{zo_pb:ass:fig}.
\end{proof}

\begin{lemma}\label{zo_pb:lemma:sg2-T}
Suppose Assumptions~\ref{zo_pb:ass:graph}--\ref{zo_pb:ass:fig} hold, and we have fixed parameters $\alpha=\kappa_1\beta$, $\beta$, and $\eta=\frac{\kappa_2}{\beta}$, where $\beta$ is large enough,
$\kappa_1>\frac{1}{\rho_2(L)}+1$ and $\kappa_2\in\Big(0,\min\{\frac{(\kappa_1-1)\rho_2(L)-1}{\rho(L)+(2\kappa_1^2+1)\rho(L^2)+1}, \frac{1}{5}\}\Big)$ are constants. Let $\{\bsx_k\}$ be the sequence generated by Algorithm~\ref{zo_pb:algorithm-random-pd}, then
\begin{subequations}
\begin{align}
\mathbb{E}[W_{k+1}]
&\le   W_{k}-\kappa_4\|\bsx_k\|^2_{\bsK}\nonumber\\
&\quad-\frac{1}{2}(\kappa_2 - 5\kappa_2^2)\Big\|\bm{v}_k+\frac{1}{\beta}\bsg_{k}^0\Big\|^2_{\bsK}\nonumber\\
&\quad-\frac{1}{8}\eta\|\bar{\bsg}^0_{k}\|_{1+\gamma}^2+\mathcal{O}(np)\eta^2 + \mathcal{O}(np^2)\eta\delta_k^2,
\label{zo_pb:zerosg:sgproof-vkLya2T}\\
\mathbb{E}[W_{4,k+1}]
&\le  W_{4,k}+2\eta L_f^2\|\bsx_k\|^2_{\bsK}-\frac{1}{8}\eta\|\bar{\bsg}_{k}^0\|^2\nonumber\\
&\quad+\mathcal{O}(p)\eta^2
+\mathcal{O}(np)\eta\delta^2_k.\label{zo_pb:zerosg:v4kspeed}
\end{align}
\end{subequations}
\end{lemma}

\begin{proof}
We provide the proof of Lemma~\ref{zo_pb:lemma:sg2-T} in the appendix.
\end{proof}

We are now ready to prove Theorem~\ref{zo_pb:algorithm-random-pd}.
\begin{proof}

Denote
\begin{align*}
\hat{V}_k=\|\bm{x}_k\|^2_{\bsK}+\Big\|\bsv_k
+\frac{1}{\beta_k}\bsg_k^0\Big\|^2_{\bsK}+n(f(\bar{x}_k)-f^*).
\end{align*}
We have
\begin{align}
&W_{k}\nonumber\\
&=\frac{1}{2}\|\bsx_{k}\|^2_{\bsK}
+\frac{1}{2}\Big\|\bsv_k+\frac{1}{\beta_k}\bsg_k^0\Big\|^2_{\bsQ+\kappa_1\bsK}\nonumber\\
&~~~+\bsx_k^\top\bsK\Big(\bm{v}_k+\frac{1}{\beta_k}\bsg_k^0\Big)+n(f(\bar{x}_k)-f^*)\nonumber\\
&\ge\frac{1}{2}\|\bsx_{k}\|^2_{\bsK}
+\frac{1}{2}\Big(\frac{1}{\rho(L)}+\kappa_1\Big)
\Big\|\bsv_k+\frac{1}{\beta_k}\bsg_k^0\Big\|^2_{\bsK}\nonumber\\
&~~~-\frac{1}{2\kappa_1}\|\bsx_{k}\|^2_{\bsK}
-\frac{1}{2}\kappa_1\Big\|\bsv_k+\frac{1}{\beta_k}\bsg_k^0\Big\|^2_{\bsK}
+n(f(\bar{x}_k)-f^*)\nonumber\\
&\ge\min\Big\{\frac{1}{2\rho(L)},~\frac{\kappa_1-1}{2\kappa_1}\Big\}\hat{V}_k\ge0,\label{zo_pb:zerosg:vkLya3}
\end{align}
Additionally, we can get $W_{k}\le(\frac{\kappa_1+1}{2}+\frac{1}{2\rho_2(L)})\hat{V}_k$.
	
Consider that $\beta=\kappa_2\sqrt{pT}/\sqrt{n}$ and $T> n^3/p$, we know that Lemma~\ref{zo_pb:lemma:sg2-T} are satisfied. So \eqref{zo_pb:zerosg:sgproof-vkLya2T} and \eqref{zo_pb:zerosg:v4kspeed} hold.
Summing~\eqref{zo_pb:zerosg:sgproof-vkLya2T} over $k \in [0, T]$ and applying~\eqref{zo_pb:zerosg:vkLya3}, we have
\begin{align}
&\frac{1}{T+1}\sum_{k=0}^{T}\mathbb{E}[\frac{1}{n}\sum_{i=1}^{n}\|x_{i,k}-\bar{x}_k\|^2]\nonumber\\
&\le\frac{1}{\kappa_4}\Big(\frac{W_{0}}{n(T+1)}
+\frac{\mathcal{O}(n)\eta\delta_k^2}{T}
+\frac{\mathcal{O}(n/p)\eta^2\kappa_\delta}{\sqrt{T(T+1)}}\Big) \nonumber\\
&=\mathcal{O}(\frac{n}{T}),
\label{zo_pb:zerosg:thm-sg-sm-equ3.1p}
\end{align}
where $W_{0} = \mathcal{O}(n)$, $\frac{W_{0}}{n(T+1)}=\mathcal{O}(\frac{1}{T})$, $\frac{n\mathcal{O}(p^2)\eta\delta_k^2}{T} = \mathcal{O}(\frac{n}{T})$, and $\frac{\mathcal{O}(n/p)\eta^2\kappa_\delta}{\sqrt{T(T+1)}}=\mathcal{O}(\frac{n}{pT})$ , which gives~\eqref{zo_pb:coro-sg-sm-equ3.1}.

From~\eqref{zo_pb:zerosg:v4kspeed},~\eqref{zo_pb:step:eta2-sm}, and~\eqref{zo_pb:zerosg:vkLya3}, summing~\eqref{zo_pb:zerosg:v4kspeed} over $k \in [0, T]$ similar to the way to get~\eqref{zo_pb:coro-sg-sm-equ3.1}, we have
\begin{align}\label{zo_pb:zerosg:thm-sg-sm-equ3p}
&\frac{1}{T+1}\sum_{k=0}^{T}\mathbb{E}[\|\nabla f(\bar{x}_k)\|_{1+\gamma}^2]=\frac{1}{n(T+1)}\sum_{k=0}^{T}\mathbb{E}[\|\bar{\bsg}_{k}^0\|_{1+\gamma}^2]\nonumber\\
&\le 8\Big(\frac{W_{4,0}}{n(T+1)\eta}
+\frac{2L_f^2}{n(T+1)}\sum_{k=0}^{T}\mathbb{E}[\|\bsx_k\|^2_{\bsK}]+\frac{\mathcal{O}(p)}{n}\nonumber\\
&~~~+\frac{\mathcal{O}(\sqrt{np})}{n\sqrt{T+1}}\Big).
\end{align}
Noting that $\eta=\kappa_2/\beta_k=\sqrt{n}/\sqrt{pT}$, and $n/T<\sqrt{p}/\sqrt{nT}$ due to $T> n^3/p$, from~\eqref{zo_pb:zerosg:thm-sg-sm-equ3p} and~\eqref{zo_pb:zerosg:thm-sg-sm-equ3.1p}, we have
\begin{align*}
\frac{1}{T}\sum_{k=0}^{T-1}\mathbb{E}[\|\nabla f(\bar{x}_k)\|_{1+\gamma}^2]
&=\mathcal{O}(\frac{\sqrt{p}}{\sqrt{nT}})
+\mathcal{O}(\frac{n}{T}),
\end{align*}
which gives~\eqref{zo_pb:coro-sg-sm-equ3}.

Summing~\eqref{zo_pb:zerosg:v4kspeed} over $ k\in[0,T]$, and using~\eqref{zo_pb:step:eta2-sm}  yield
\begin{align}\label{zo_pb:zerosg:thm-sg-sm-equ4p}
&n(\mathbb{E}[f(\bar{x}_{T+1})]-f^*)=\mathbb{E}[W_{4,T+1}]\nonumber\\
&\le W_{4,0}+\frac{2\sqrt{n}}{\sqrt{pT}} L_f^2\sum_{k=0}^{T}\|\bsx_k\|^2_{\bsK}+n\mathcal{O}(p)\eta^2\frac{T+1}{T}\nonumber\\
&~~~+\mathcal{O}(np)\eta\delta^2_k\sqrt{\frac{T+1}{T}}.
\end{align}

Noting that $W_{4,0}=\mathcal{O}(n)$ and $\sqrt{n}n/\sqrt{pT}<1$ due to $T> n^3/p$, from \eqref{zo_pb:zerosg:thm-sg-sm-equ3.1p} and~\eqref{zo_pb:zerosg:thm-sg-sm-equ4p}, we have $\mathbb{E}[f(\bar{x}_{T+1})]-f^*=\mathcal{O}(1)$, which gives~\eqref{zo_pb:coro-sg-sm-equ4}.

\end{proof}

\begin{remark}
Theorem~\ref{zo_pb:thm-sg-smT} indicates that we improve the convergence rate of the algorithm in~\cite{zhang2021convergence} from $\mathcal{O}(\frac{\sqrt{p}}{\sqrt{T}})$ to $\mathcal{O}(\frac{\sqrt{p}}{\sqrt{nT}})$, which is the same convergence rate of the algorithms proposed in~\cite{yi2021zerothorder}.
\end{remark}

\section{Numerical Experiments}\label{zo_pb:sec:exp}

\subsection{Black-box binary classification}\label{zo_pb:sec:exp-bc}

We consider a non-linear least square problem \cite{xu2020second, liu2019signsgd, liu2018zeroth}, i.e., problem with $f_i(\mathbf x) = \left ( y_i - \phi(\mathbf x; \mathbf a_i) \right )^2$ for $i\in[n]$, where $\phi(\mathbf x; \mathbf a_i) = \frac{1}{1 + e^{-\mathbf a_i^{T} \mathbf x}}$. For preparing the synthetic dataset, we randomly draw samples $\mathbf a_i$ from $\mathcal{N}(\mathbf 0, \mathbf{I} )$, and we set a optimal vector $\mathbf{x_{opt}} = \mathbf{1}$, the label is $y_i = 1$ if $\phi(\mathbf x_{opt}; \mathbf a_i) \geq 0.5$ and $0$ otherwise. The training set has $200$ samples per agent and test set has $10,000$ samples. We set the dimension $d$ of $\mathbf a_i$ as $100$, batchsize is $1$, and the total iteration number as $500$. As suggested in the work \cite{liu2019signsgd}, the smooth parameter $\delta = \frac{10}{\sqrt{Td}}$.

We compare the proposed algorithms with ZODIAC only since ZODIAC achieves better result than other state-of-the-art algorithms.
The communication topology of $N = 500$ agents is generated randomly following the Erd\H{o}s - R\' enyi model with probability of $1.01 \log (N)/N$ in Figure~\ref{fig:graph_500}. 
The training loss and testing accuracy are shown in Figure~\ref{fig:loss_500} and Figure~\ref{fig:acc_500} respectively.
We can easily see that the proposed algorithm converges faster than ZODIAC and returns a better result in terms of testing accuracy, shown in Table~\ref{tab:acc}.

\begin{figure}[!ht]
\centering
  \includegraphics[width=0.45\textwidth]{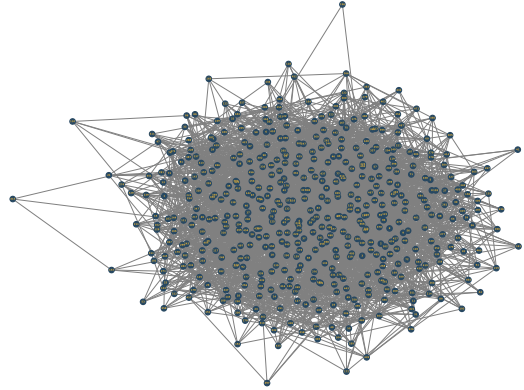}
  \caption{Communication topology of 500 agents.}
  \label{fig:graph_500}
\end{figure}

\begin{figure}[!ht]
\centering
  \includegraphics[width=0.45\textwidth]{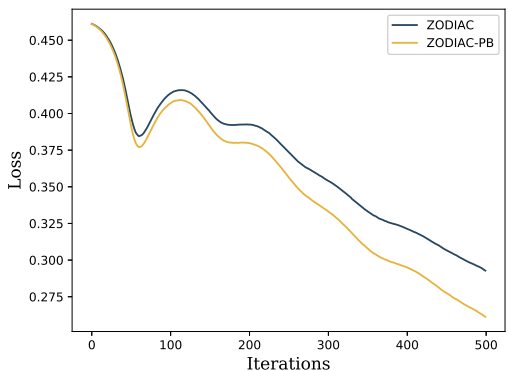}
  \caption{Training Loss.}
  \label{fig:loss_500}
\end{figure}

\begin{figure}[!ht]
\centering
  \includegraphics[width=0.45\textwidth]{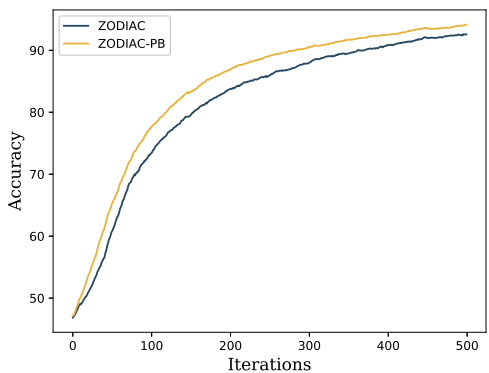}
  \caption{Testing Accuracy.}
  \label{fig:acc_500}
\end{figure}

\begin{table}[!ht]
\caption{Accuracy}
\label{tab:acc}
\begin{center}
\begin{tabular}{cc}
\multicolumn{1}{c}{Algorithm}  & Accuracy($\%$)
\\ \hline
ZODIAC-PB       & \textbf{94.15}  \\
ZODIAC       & 92.56  \\
\hline
\end{tabular}
\end{center}
\end{table}

\subsection{Generation of adversarial examples from black-box DNNs}\label{zo_pb:sec:exp-dnn}

We consider the benchmark example of generation of adversarial examples from black-box DNNs in ZO optimization literature \cite{chen2017zoo, liu2019signsgd, liu2018zeroth}.
In image classification tasks, convolutional neural networks are vulnerable to adversarial examples \cite{goodfellow2014explaining} even under small perturbations, which leads misclassifications.
Considering the setting of zeroth-order attacks \cite{carlini2017towards}, the model is hidden and no gradient information is available. We treat this task of generating adversarial examples as an zeroth-order optimization problem.

Formally, the loss function is given as in~\eqref{zo_pb:exp:loss_gan} 
\begin{equation}\label{zo_pb:exp:loss_gan}
\begin{split}
f_i(\mathbf x) = & c \cdot \max \{ F_{y_i} (0.5 \cdot \tanh ( \tanh^{-1} 2 \mathbf  a_i + \mathbf x)) \\
&-\max_{j \neq y_i} F_j(0.5 \cdot \tanh ( \tanh^{-1} 2 \mathbf a_i + \mathbf x)) , 0 \} \\
& + \|0.5 \cdot \tanh ( \tanh^{-1} 2 \mathbf a_i + \mathbf x) - \mathbf a_i \|_2^2
\end{split}
\end{equation}
where $(\mathbf a_i,y_i)$ denotes the pair of the $i$th natural image $\mathbf a_i$ and its original class label $y_i$. The output of function $F(\mathbf z)=[F_1(\mathbf z),\ldots,F_N(\mathbf z)]$ is the well-trained model prediction of the input $\mathbf z$ in all $N$ image classes.
The well-trained DNN model \footnote{\url{https://github.com/carlini/nn_robust_attacks}} on MNIST handwritten has $99.4\%$ test accuracy on natural examples \cite{liu2018zeroth}.
The purpose of this experiment is to generate \textit{false} examples to attack the DNN model in order to have a wrong prediction, i.e. if feeding an original image with label $1$, the DNN predicts it as $1$, however after generating the \textit{false} example based on the original $1$, the DNN should make a wrong prediction.
We conduct two experiments on 10 agents and 50 agents scenarios.
\subsubsection{10 agents}

We compare the proposed algorithm with several existing algorithms, namely ZODIAC~\cite{zhang2021convergence}, ZODPDA~\cite{yi2021zerothorder}, ZO-GDA  \cite{tang2020distributedzero}, and ZONE-M~\cite{hajinezhad2019zone} on a communication topology with 10 agents following the Erd\H{o}s - R\' enyi model with probability of $0.4$.
The digit we consider to attack is $4$.
Additionally, we compare with centralized ZO algorithms, namely ZO-SCD~\cite{lian2016Comprehensive}, and ZO-SGD ~\cite{ghadimi2013stochastic} as baselines.
The training loss is shown in Figure.~\ref{fig:loss_10_agent} and the distortion of the generated examples is shown in Table~\ref{zo_fb:tab:dist10}.
we can conclude that the ZODIAC-PB outperformed among all the algorithms compared in the literature.

\begin{figure}[!ht]
\centering
  \includegraphics[width=0.45\textwidth]{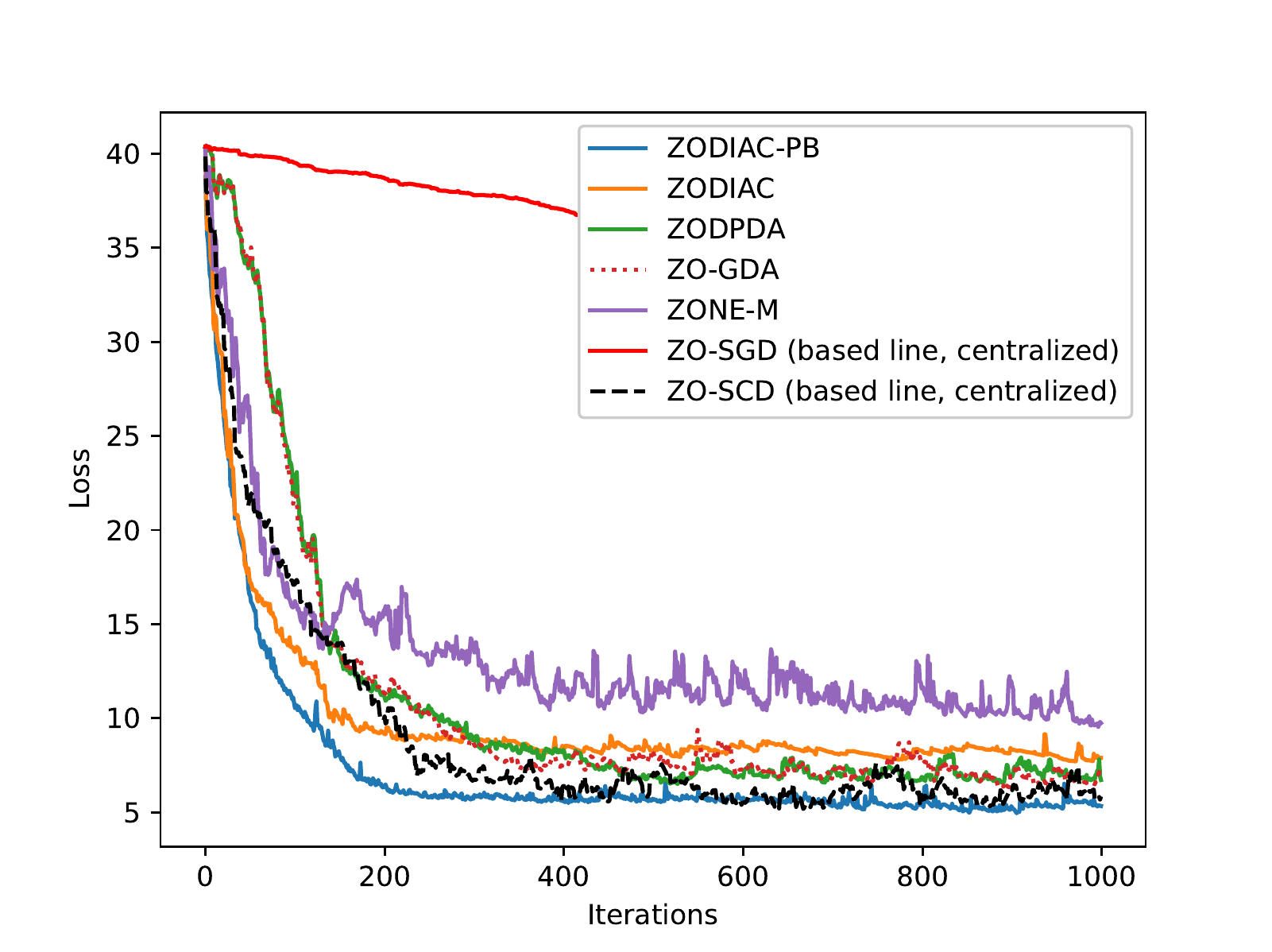}
  \caption{Performance comparison of training loss for 10 agents.}
  \label{fig:loss_10_agent}
\end{figure}

\begin{table}[ht!]
\caption{Distortion (10 agents)}
\label{zo_fb:tab:dist10}
\begin{center}
\begin{tabular}{cc}
\multicolumn{1}{c}{Algorithm}  & $l_2$ Distortion
\\ \hline
ZODIAC-PB   & \textbf{4.92} \\
ZODIAC      & 7.18  \\
ZODPDA       & 6.44  \\
ZO-GDA       & 7.23  \\
ZONE-M       & 9.96  \\
ZO-SGD        & 5.69  \\
ZO-SCD        & 5.14  \\
\hline
\end{tabular}
\end{center}
\end{table}

\subsubsection{50 agents}
In this case, we only compare the proposed algorithms with ZODIAC to attack digit $0$ since since ZODIAC achieves better result than other state-of-the-art algorithms.
We tested them on 50 agents respectively, the topology are shown in Figure.~\ref{fig:graph_50}. The graphs are generated randomely following the Erd\H{o}s - R\' enyi model with probability of $0.4$.

The distortion of the generated examples is shown in Table~\ref{zo_fb:tab:dist50}, and the generated examples and prediction results are shown in Table~\ref{table:digit0}.
In this experiment, we can conclude that the proposed algorithm accelerate the convergence in Figure~\ref{fig:comparison}.
Moreover, the distortion generated from ZODIAC-PB is 5.67, which is around $34.7\%$ improvement.

\begin{figure}[!ht]
\centering
  \includegraphics[width=0.45\textwidth]{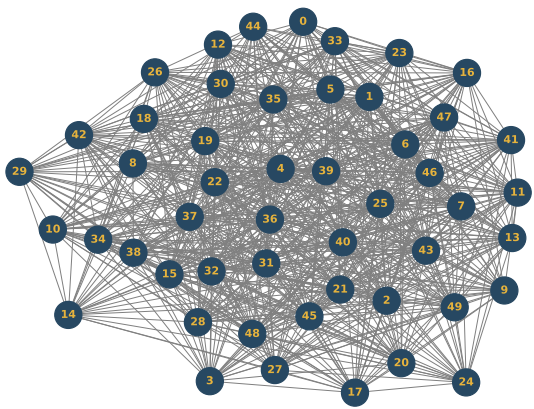}
  \caption{Communication topology of 50 agents.}
  \label{fig:graph_50}
\end{figure}

\begin{figure}[!ht]
\centering
  \includegraphics[width=0.45\textwidth]{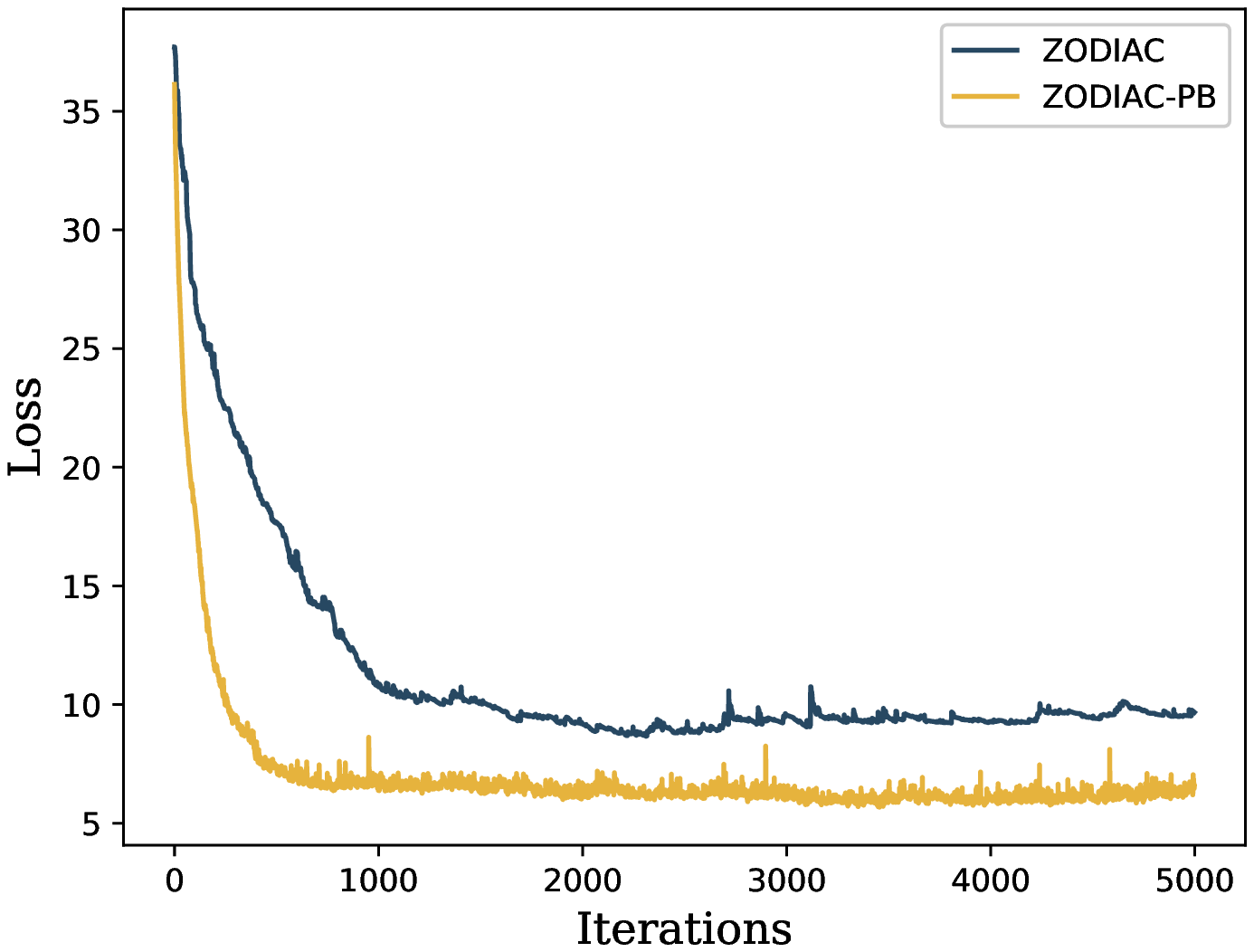}
  \caption{Performance comparison of training loss for 50 agents.}
  \label{fig:comparison}
\end{figure}


\begin{table}[!ht]
\caption{Distortion (50 agents)}
\label{zo_fb:tab:dist50}
\begin{center}
\begin{tabular}{cc}
\multicolumn{1}{c}{Algorithm}  & $l_2$ Distortion
\\ \hline
ZODIAC-PB       & \textbf{5.67}  \\
ZODIAC       & 8.68  \\
\hline
\end{tabular}
\end{center}
\end{table}

\begin{table*}
  [!ht] \caption{Comparison of generated adversarial examples from a black-box DNN on MNIST: digit class ``0''.} \label{table:digit0}
  \begin{center}
  \begin{adjustbox}{max width=0.75\textwidth}
  \begin{tabular}
      {ccccccccccc}
      \hline
      	Image ID & 3 & 10 & 13 & 25 & 28 & 55 & 69 & 71 & 101 & 126 \\
      \hline &&&&&&&&&& \vspace{-0.2cm} \\
      	\red{Original} &
        \parbox[c]{2.2em}{\includegraphics[width=0.4in]{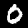}} &
        \parbox[c]{2.2em}{\includegraphics[width=0.4in]{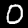}} &
        \parbox[c]{2.2em}{\includegraphics[width=0.4in]{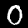}} &
        \parbox[c]{2.2em}{\includegraphics[width=0.4in]{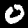}} &
        \parbox[c]{2.2em}{\includegraphics[width=0.4in]{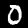}} &
        \parbox[c]{2.2em}{\includegraphics[width=0.4in]{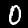}} &
        \parbox[c]{2.2em}{\includegraphics[width=0.4in]{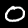}} &
        \parbox[c]{2.2em}{\includegraphics[width=0.4in]{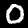}} &
        \parbox[c]{2.2em}{\includegraphics[width=0.4in]{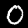}} &
        \parbox[c]{2.2em}{\includegraphics[width=0.4in]{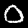}} \\
        \red{Classified as} & \red{0} & \red{0} & \red{0} & \red{0} & \red{0} & \red{0} & \red{0} & \red{0} & \red{0} & \red{0} \\
        
	\hline &&&&&&&&&& \vspace{-0.2cm} \\
      	\thead{ZODIAC} &
        \parbox[c]{2.2em}{\includegraphics[width=0.4in]{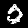}} &
        \parbox[c]{2.2em}{\includegraphics[width=0.4in]{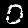}} &
        \parbox[c]{2.2em}{\includegraphics[width=0.4in]{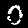}} &
        \parbox[c]{2.2em}{\includegraphics[width=0.4in]{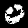}} &
        \parbox[c]{2.2em}{\includegraphics[width=0.4in]{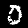}} &
        \parbox[c]{2.2em}{\includegraphics[width=0.4in]{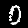}} &
        \parbox[c]{2.2em}{\includegraphics[width=0.4in]{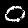}} &
        \parbox[c]{2.2em}{\includegraphics[width=0.4in]{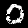}} &
        \parbox[c]{2.2em}{\includegraphics[width=0.4in]{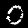}} &
        \parbox[c]{2.2em}{\includegraphics[width=0.4in]{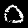}} \\
        Classified as & 9 & 2 & 2 & 9 & 3 & 7 & 9 & 2 & 5 & 2 \\        
        
	\hline &&&&&&&&&& \vspace{-0.2cm} \\
      	\thead{ZODIAC-PB} &
        \parbox[c]{2.2em}{\includegraphics[width=0.4in]{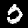}} &
        \parbox[c]{2.2em}{\includegraphics[width=0.4in]{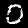}} &
        \parbox[c]{2.2em}{\includegraphics[width=0.4in]{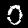}} &
        \parbox[c]{2.2em}{\includegraphics[width=0.4in]{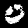}} &
        \parbox[c]{2.2em}{\includegraphics[width=0.4in]{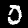}} &
        \parbox[c]{2.2em}{\includegraphics[width=0.4in]{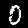}} &
        \parbox[c]{2.2em}{\includegraphics[width=0.4in]{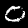}} &
        \parbox[c]{2.2em}{\includegraphics[width=0.4in]{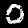}} &
        \parbox[c]{2.2em}{\includegraphics[width=0.4in]{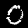}} &
        \parbox[c]{2.2em}{\includegraphics[width=0.4in]{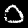}} \\
        Classified as & 5 & 2 & 5 & 5 & 3 & 2 & 9 & 2 & 5 & 2 \\

      \hline
  \end{tabular}
  \end{adjustbox}
  \end{center}
\end{table*}

\section{Conclusions}\label{zo_pb:sec:conclusions}

In this paper, we investigated the acceleration of ZO stochastic distributed nonconvex optimization problems and proposed ZODIAC-PB based on the primal--dual framework.
We demonstrated that the proposed algorithm achieves the convergence rate of $\mathcal{O}(\sqrt{p}/\sqrt{nT})$ for general nonconvex cost functions.
Additionally, we illustrated the efficacy of ZODIAC-PB through benchmark examples on a large-scale multi-agent topology in comparison with the existing state-of-the-art centralized and distributed ZO algorithms.

\section*{Acknowledgments}
The authors would like to thank Dr. Xinlei Yi and Mr. Yunlong Dong for their fruitful discussions on this work.

\bibliographystyle{IEEEtran}
\bibliography{zeroth_order}

\appendix\label{zo_pb:sec:app}

\section*{Proof of Lemma~\ref{zo_pb:lemma:sg2-T}}\label{zo_pb:proof:lemma4}
\begin{proof}

Consider the following Lyapunov candidate function 
\begin{align}
W_{k}= &\underbrace{\frac{1}{2}\|\bsx_{k}\|^2_{\bsK}}_{W_{1, k}} + \underbrace{\frac{1}{2}\Big\|\bsv_k+\frac{1}{\beta}\bsg_k^0\Big\|^2_{\bsQ+\kappa_1\bsK}}_{W_{2, k}} \nonumber \\ &+ \underbrace{\bsx_k^\top\bsK\Big(\bm{v}_k+\frac{1}{\beta}\bsg_k^0\Big)}_{W_{3, k}} + \underbrace{n(f(\bar{x}_k)-f^*)}_{W_{4, k}}
\end{align}
where $\bsQ=R\Lambda^{-1}_1R^{\top}\otimes {\bm I}_p$.
Additionally, we denote $g^s_{i,k}=\nabla \mathbb{E}[f_i(x+\delta_{i,k} e_{i})]$, $\bsg^s_k=\col(g^s_{1,k},\dots,g^s_{n,k})$, $\bar{\bsg}^s_k=\bsH\bsg^s_k$, $\bar{g}^e_k=\frac{1}{n}({\bm 1}_n^\top\otimes{\bm I}_p)\bsg^e_k$, and $\bar{\bsg}^e_k={\bm 1}_n\otimes\bar{g}^e_k=\bsH\bsg^e_k$.

(i) We have
\begin{align}
&\mathbb{E}[W_{1,k+1}]
=\mathbb{E}\Big[\frac{1}{2}\|\bm{x}_{k+1} \|^2_{\bsK}\Big]\nonumber\\
&\overset{\mathrm{Eq.~\ref{zo_pb:alg:random-pd-x}}}{=\joinrel=}\mathbb{E}\Big[\frac{1}{2}\|\bm{x}_k-\eta(\alpha\bsL\bm{x}_k+\beta\bm{v}_k+\sigma(\bsg^e_k, \gamma)) \|^2_{\bsK}\Big]\nonumber\\
&\overset{\text{(a)}}{=}\mathbb{E}\Big[\frac{1}{2}\|\bm{x}_k\|^2_{\bsK}-\eta\alpha\|\bsx_k\|^2_{\bsL}
+\frac{1}{2}\eta^2\alpha^2\|\bsx_k\|^2_{\bsL^2}
\nonumber\\
&~~~-\eta\beta\bsx^\top_k({\bm I}_{np}-\eta\alpha\bsL)\bsK\Big(\bm{v}_k+\frac{1}{\beta}\sigma(\bsg^e_k, \gamma)\Big)\nonumber\\
&~~~+\frac{1}{2}\eta^2\beta^2\Big\|\bm{v}_k+\frac{1}{\beta}\sigma(\bsg^e_k, \gamma)\Big\|^2_{\bsK}\Big]\nonumber\\
&\overset{\text{(b)}}{=}W_{1,k}-\|\bsx_k\|^2_{\eta\alpha\bsL
-\frac{1}{2}\eta^2\alpha^2\bsL^2}\nonumber\\
&~~~-\eta\beta\bsx^\top_k({\bm I}_{np}-\eta\alpha\bsL)\bsK\Big(\bm{v}_k
+\frac{1}{\beta}\bsg^s_k\Big)\nonumber\\
&~~~+\frac{1}{2}\eta^2\beta^2\mathbb{E}\Big[\Big\|\bm{v}_k+\frac{1}{\beta}\bsg_k^0
+\frac{1}{\beta}\sigma(\bsg^e_k, \gamma)-\frac{1}{\beta}\bsg_k^0\Big\|^2_{\bsK}\Big]\nonumber\\
&\overset{\text{(c)}}{\le} W_{1,k}-\|\bsx_k\|^2_{\eta\alpha\bsL
-\frac{1}{2}\eta^2\alpha^2\bsL^2}\nonumber\\
&~~~-\eta\beta\bsx^\top_k\bsK\Big(\bm{v}_k+\frac{1}{\beta}\bsg_k^0\Big)\nonumber\\
&~~~+\frac{1}{2}\eta\|\bm{x}_k\|^2_{\bsK}
+\frac{1}{2}\eta\|\bsg^s_k-\bsg_k^0\|^2\nonumber\\
&~~~+\frac{1}{2}\eta^2\alpha^2\|\bm{x}_k\|^2_{\bsL^2}
+\frac{1}{2}\eta^2\beta^2\Big\|\bm{v}_k+\frac{1}{\beta}\bsg_k^0\Big\|^2_{\bsK}\nonumber\\
&~~~+\frac{1}{2}\eta^2\alpha^2\|\bm{x}_k\|^2_{\bsL^2}
+\frac{1}{2}\eta^2\|\bsg^s_k-\bsg_k^0\|^2\nonumber\\
&~~~+\eta^2\beta^2\Big\|\bm{v}_k+\frac{1}{\beta}\bsg_k^0\Big\|^2_{\bsK}
+\eta^2\mathbb{E}[\|\sigma(\bsg^e_k, \gamma)-\bsg_k^0\|^2]\nonumber\\
&\overset{\text{(d)}}{\le} W_{1,k}-\|\bsx_k\|^2_{\eta\alpha\bsL-\frac{1}{2}\eta\bsK
-\frac{3}{2}\eta^2\alpha^2\bsL^2-\eta(1+5\eta)L_f^2\bsK}
\nonumber\\
&~~~-\eta\beta\bsx^\top_k\bsK\Big(\bm{v}_k+\frac{1}{\beta}\bsg_k^0\Big)
+\Big\|\bm{v}_k+\frac{1}{\beta}\bsg_k^0\Big\|^2_{\frac{3}{2}\eta^2\beta^2\bsK}\nonumber\\
&~~~+nL_f^2\eta\Big[\frac{p}{4}+(\frac{p}{4}+4)\eta\Big]\delta_k^2+2\eta^2\mathbb{E}[\|\sigma(\bsg^e_k, \gamma)\|^2],\label{zo_pb:zerosg:v1k}
\end{align}
where (a) holds due to Lemma~1 and~2 in~\cite{Yi2018distributed}; (b) holds due to $\mathbb{E}[\bsg^e_k] = \bsg^s_k$ and that $x_{i,k}$ and $v_{i,k}$ are independent of $u_{i,k}$ and $\xi_{i,k}$; (c) holds due to the Cauchy--Schwarz inequality and $\rho(\bsK)=1$; (d) holds due to $\|\bsg^s_k-\bsg_k^0\|^2\le 2L_f^2\|\bsx_{k}\|^2_{\bsK}+\frac{np}{2}L_f^2\delta_k^2$ and $\mathbb{E}[\|\bsg_k^0-\sigma(\bsg^e_k, \gamma)\|^2]\le 4L_f^2\|\bsx_{k}\|^2_{\bsK}+4nL_f^2\delta_k^2+2\mathbb{E}[\|\sigma(\bsg^e_k, \gamma)\|^2]$.

(ii)
\begin{align}
W_{2,k+1}
=\frac{1}{2}\Big\|\bsv_{k+1}+\frac{1}{\beta}\bsg_{k+1}^0\Big\|^2_{\bsQ+\kappa_1\bsK}\label{zo_pb:zerosg:v2k-1}
\end{align}
\begin{align}
&\overset{\mathrm{Eq.~\ref{zo_pb:alg:random-pd-q}}}{=\joinrel=}\frac{1}{2}\Big\|\bm{v}_k+\frac{1}{\beta}\bsg_{k}^0+\eta\beta\bsL\bm{x}_k
+\frac{1}{\beta}(\bsg_{k+1}^0-\bsg_{k}^0) \Big\|^2_{\bsQ+\kappa_1\bsK}\nonumber\\
&\overset{\text{(e)}}{=}W_{2,k}+\eta\beta\bsx^\top_k(\bsK+\kappa_1\bsL)\Big(\bm{v}_k+\frac{1}{\beta}\bsg_k^0\Big)\nonumber\\
&~~~+\|\bsx_k\|^2_{\frac{1}{2}\eta^2\beta^2(\bsL+\kappa_1\bsL^2)}
+\frac{1}{2\beta^2}\Big\|\bsg_{k+1}^0-\bsg_{k}^0\Big\|^2_{\bsQ+\kappa_1\bsK}\nonumber\\
&~~~+\frac{1}{\beta}\Big(\bm{v}_k+\frac{1}{\beta}\bsg_{k}^0
+\eta\beta\bsL\bm{x}_k\Big)^\top(\bsQ
+\kappa_1\bsK)(\bsg_{k+1}^0-\bsg_{k}^0)\nonumber\\
&\overset{\text{(f)}}{\le} W_{2,k}+\eta\beta\bsx^\top_k(\bsK+\kappa_1\bsL)\Big(\bm{v}_k+\frac{1}{\beta}\bsg_k^0\Big)\nonumber\\
&~~~+\|\bsx_k\|^2_{\frac{1}{2}\eta^2\beta^2(\bsL+\kappa_1\bsL^2)}
+\frac{1}{2\beta^2}\|\bsg_{k+1}^0-\bsg_{k}^0\|^2_{\bsQ+\kappa_1\bsK}\nonumber\\
&~~~+\frac{\eta}{2}\Big\|\bm{v}_k+\frac{1}{\beta}\bsg_{k}^0\Big\|^2_{\bsQ+\kappa_1\bsK}
+\frac{1}{2\eta\beta^2}\|\bsg_{k+1}^0-\bsg_{k}^0\|^2_{\bsQ+\kappa_1\bsK}\nonumber\\
&~~~+\frac{1}{2}\eta^2\beta^2\|\bsL\bm{x}_k\|^2_{\bsQ+\kappa_1\bsK}
+\frac{1}{2\beta^2}\|\bsg_{k+1}^0-\bsg_{k}^0\|^2_{\bsQ+\kappa_1\bsK}\nonumber\\
&\overset{\text{(g)}}{=} W_{2,k}+\eta\beta\bsx^\top_k(\bsK+\kappa_1\bsL)\Big(\bm{v}_k+\frac{1}{\beta}\bsg_k^0\Big)\nonumber\\
&~~~+\|\bsx_k\|^2_{\eta^2\beta^2(\bsL+\kappa_1\bsL^2)}
+\Big\|\bm{v}_k+\frac{1}{\beta}\bsg_{k}^0
\Big\|^2_{\frac{1}{2}\eta(\bsQ+\kappa_1\bsK)}\nonumber\\
&~~~+\frac{1}{\beta^2}\Big(1+\frac{1}{2\eta}\Big)
\|\bsg_{k+1}^0-\bsg_{k}^0\|^2_{\bsQ+\kappa_1\bsK}\nonumber\\
&\overset{\text{(h)}}{\le} W_{2,k}+\eta\beta\bsx^\top_k(\bsK+\kappa_1\bsL)\Big(\bm{v}_k+\frac{1}{\beta}\bsg_k^0\Big)\nonumber\\
&~~~+\|\bsx_k\|^2_{\eta^2\beta^2(\bsL+\kappa_1\bsL^2)}
+\Big\|\bm{v}_k+\frac{1}{\beta}\bsg_{k}^0
\Big\|^2_{\frac{1}{2}\eta(\bsQ+\kappa_1\bsK)}\nonumber\\
&~~~+\frac{1}{\beta^2}\Big(1+\frac{1}{2\eta}\Big)\Big(\frac{1}{\rho_2(L)}+\kappa_1\Big)
\|\bsg_{k+1}^0-\bsg_{k}^0\|^2\nonumber\\
&\overset{\text{(i)}}{\le} W_{2,k}+\eta\beta\bsx^\top_k(\bsK+\kappa_1\bsL)\Big(\bm{v}_k+\frac{1}{\beta}\bsg_k^0\Big)\nonumber\\
&~~~+\|\bsx_k\|^2_{\eta^2\beta^2(\bsL+\kappa_1\bsL^2)}
+\Big\|\bm{v}_k+\frac{1}{\beta}\bsg_{k}^0
\Big\|^2_{\frac{1}{2}\eta(\bsQ+\kappa_1\bsK)}\nonumber\\
&~~~+\frac{\eta}{\beta^2}\Big(\eta+\frac{1}{2}\Big)
\Big(\frac{1}{\rho_2(L)}+\kappa_1\Big)L_f^2\|\bar{\bsg}^e_{k}\|^2,\label{zo_pb:zerosg:v2k-2}
\end{align}
where (e) holds due to (a) holds due to Lemma~1 and~2 in~\cite{Yi2018distributed}; (f) holds due to the Cauchy--Schwarz inequality; (g) holds due to Lemma~1 and~2 in~\cite{Yi2018distributed}; (h) holds due to $\rho(\bsQ+\kappa_1\bsK)\le\rho(\bsQ)+\kappa_1\rho(\bsK)$, and $\rho(\bsK)=1$; (i) holds due to $\|\bsg^0_{k+1}-\bsg^0_{k}\|^2\le \eta^2L_f^2\|\bar{\bsg}^e_{k}\|^2
\le\eta^2L_f^2\|\bsg^e_{k}\|^2$.

Moreover, we have the following two inequalities hold:
\begin{equation}\label{zo_pb:zerosg:v2k-3}
\|\bsg_{k+1}^0\|^2_{\bsQ+\kappa_1\bsK}
\le\Big(\frac{1}{\rho_2(L)}+\kappa_1\Big)\|\bsg_{k+1}^0\|^2.
\end{equation}

\begin{align}\label{zo_pb:zerosg:v2k-4}
\Big\|\bm{v}_k+\frac{1}{\beta_k}\bsg_{k}^0\Big\|^2_{\bsQ+\kappa_1\bsK}
\le\Big(\frac{1}{\rho_2(L)}+\kappa_1\Big)\Big\|\bm{v}_k+\frac{1}{\beta_k}\bsg_{k}^0
\Big\|^2_{\bsK}.
\end{align}

Then, from \eqref{zo_pb:zerosg:v2k-1}--\eqref{zo_pb:zerosg:v2k-4}, we have
\begin{align}
&W_{2,k+1}\nonumber\\
&\le W_{2,k}
+\eta\beta\bsx^\top_k(\bsK+\kappa_1\bsL)\Big(\bm{v}_k+\frac{1}{\beta}\bsg_k^0\Big)\nonumber\\
&~~~+\frac{1}{2}\eta\Big(\frac{1}{\rho_2(L)}+\kappa_1\Big)
\Big\|\bm{v}_k+\frac{1}{\beta}\bsg_{k}^0\Big\|^2_{\bsK}\nonumber\\
&~~~+\|\bsx_k\|^2_{\eta^2\beta^2(\bsL+\kappa_1\bsL^2)}\nonumber\\
&~~~+\frac{\eta}{\beta^2}\Big(\eta+\frac{1}{2}\Big)
\Big(\frac{1}{\rho_2(L)}+\kappa_1\Big)L_f^2\|\bar{\bsg}^e_{k}\|^2.
\label{zo_pb:zerosg:v2k}
\end{align}

(iii) We have
\begin{align}
W_{3,k+1}=\bsx_{k+1}^\top\bsK\Big(\bm{v}_{k+1}+\frac{1}{\beta}\bsg_{k+1}^0\Big).\label{zo_pb:zerosg:v3k-1}
\end{align}

\begin{align}
&\mathbb{E}\Big[W_{3,k+1}\Big]\nonumber\\
&\overset{\mathrm{Eq.~\ref{zo_pb:alg:random-pd}}}{=\joinrel=}\mathbb{E}\Big[(\bm{x}_k-\eta(\alpha\bsL\bm{x}_k+\beta\bm{v}_k
+\bsg_k^0+\sigma(\bsg^e_k, \gamma)-\bsg_k^0))^\top
\nonumber\\
&~~~\bsK\Big(\bm{v}_k+\frac{1}{\beta}\bsg_{k}^0+\eta\beta\bsL\bm{x}_k
+\frac{1}{\beta}(\bsg_{k+1}^0-\bsg_{k}^0)\Big)\Big]\nonumber\\
&\overset{\text{(j)}}{=}\bm{x}_k^\top(\bsK-\eta(\alpha+\eta\beta^2)\bsL)\Big(\bm{v}_k+\frac{1}{\beta}\bsg_{k}^0\Big)+\|\bm{x}_k\|^2_{\eta\beta(\bsL-\eta\alpha\bsL^2)}\nonumber\\
&~~~+\frac{1}{\beta}\bm{x}_k^\top(\bsK-\eta\alpha\bsL)
\mathbb{E}[\bsg_{k+1}^0-\bsg_{k}^0]-\eta\beta\Big\|\bm{v}_k+\frac{1}{\beta}\bsg_{k}^0\Big\|^2_{\bsK}\nonumber\\
&~~~-\eta\Big(\bm{v}_k+\frac{1}{\beta}\bsg_{k}^0\Big)^\top\bsK
\mathbb{E}[\bsg_{k+1}^0-\bsg_{k}^0]\nonumber\\
&~~~-\eta(\bsg_k^s-\bsg_k^0)^\top
\bsK\Big(\bm{v}_k+\frac{1}{\beta}\bsg_{k}^0+\eta\beta\bsL\bm{x}_k\Big)\nonumber\\
&~~~-\frac{1}{\beta}\mathbb{E}[\eta(\sigma(\bsg^e_k, \gamma)-\bsg_k^0)^\top
\bsK(\bsg_{k+1}^0-\bsg_{k}^0)]\nonumber\\
&\overset{\text{(k)}}{\le}\bm{x}_k^\top(\bsK-\eta\alpha\bsL)\Big(\bm{v}_k+\frac{1}{\beta}\bsg_{k}^0\Big)
+\frac{1}{2}\eta^2\beta^2\|\bsL\bsx_k\|^2\nonumber\\
&~~~+\frac{1}{2}\eta^2\beta^2\Big\|\bm{v}_k+\frac{1}{\beta}\bsg_{k}^0\Big\|^2_{\bsK}
+\|\bm{x}_k\|^2_{\eta\beta(\bsL-\eta\alpha\bsL^2)}\nonumber\\
&~~~+\frac{1}{2}\eta\|\bm{x}_k\|^2_\bsK
+\frac{1}{2\eta\beta^2}\mathbb{E}[\|\bsg_{k+1}^0-\bsg_{k}^0\|^2\nonumber\\
&~~~+\frac{1}{2}\eta^2\alpha^2\|\bsL\bm{x}_k\|^2
+\frac{1}{2\beta^2}\mathbb{E}[\|\bsg_{k+1}^0-\bsg_{k}^0\|^2]\nonumber\\
&~~~
-\eta\beta\Big\|\bm{v}_k+\frac{1}{\beta}\bsg_{k}^0\Big\|^2_{\bsK}\nonumber\\
&~~~+\frac{1}{2}\eta^2\beta^2\Big\|\bm{v}_k+\frac{1}{\beta}\bsg_{k}^0\Big\|^2_{\bsK}
+\frac{1}{2\beta^2}\mathbb{E}[\|\bsg_{k+1}^0-\bsg_{k}^0\|^2]\nonumber\\
&~~~+\frac{1}{2}\eta\|\bsg_k^s-\bsg_k^0\|^2
+\frac{1}{2}\eta\Big\|\bm{v}_k+\frac{1}{\beta}\bsg_{k}^0\Big\|^2_{\bsK}\nonumber\\
&~~~
+\frac{1}{2}\eta^2\|\bsg_k^s-\bsg_k^0\|^2
+\frac{1}{2}\eta^2\beta^2\|\bsL\bm{x}_k\|^2\nonumber\\
&~~~
+\frac{1}{2}\eta^2\mathbb{E}[\|\sigma(\bsg^e_k, \gamma)-\bsg_k^0\|^2]
+\frac{1}{2\beta^2}\mathbb{E}[\|\bsg_{k+1}^0-\bsg_{k}^0\|^2]\nonumber\\
&\overset{\text{(l)}}{\le} W_{3,k}
-\eta\alpha\bm{x}_k^\top\bsL\Big(\bm{v}_k+\frac{1}{\beta}\bsg_{k}^0\Big)\nonumber\\
&~~~+\|\bm{x}_k\|^2_{\eta(\beta\bsL+\frac{1}{2}\bsK)
+\eta^2(\frac{1}{2}\alpha^2-\alpha\beta+\beta^2)\bsL^2
+\eta(1+3\eta)L_f^2\bsK}\nonumber\\
&~~~+\eta^2\Big[1 + (\frac{1}{2\eta\beta^2}+\frac{3}{2\beta^2}\Big)L_f^2\Big]\mathbb{E}[\|\sigma(\bsg^e_k, \gamma)\|^2]\nonumber\\
&~~~+nL_f^2\eta\Big[\frac{p}{4}+(\frac{p}{4}+2)\eta\Big]\delta^2_k\nonumber\\
&~~~-\Big\|\bm{v}_k+\frac{1}{\beta}\bsg_{k}^0\Big\|^2_{\eta(\beta-\frac{1}{2}
-\eta\beta^2)\bsK}.
\label{zerosg:v3k-2}
\end{align}

where (j) holds since $K_nL=LK_n=L$, $\mathbb{E}[\bsg^e_k] = \bsg^s_k$, and that $x_{i,k}$ and $v_{i,k}$ are independent; (k) holds due to the Cauchy--Schwarz inequality, the Jensen's inequality, and $\rho(\bsK)=1$; (l) holds due to $\|\bsg^s_k-\bsg_k^0\|^2\le 2L_f^2\|\bsx_{k}\|^2_{\bsK}+\frac{np}{2}L_f^2\delta_k^2$ and $\mathbb{E}[\|\bsg_k^0-\bsg^e_k\|^2]\le 4L_f^2\|\bsx_{k}\|^2_{\bsK}+4nL_f^2\delta_k^2+2\mathbb{E}[\|\sigma(\bsg^e_k, \gamma)\|^2]$..

(iv) We have
\begin{align}
&\mathbb{E}[W_{4,k+1}]=\mathbb{E}[n(f(\bar{x}_{k+1})-f^*)]
=\mathbb{E}[\tilde{f}(\bar{\bsx}_{k+1})-nf^*]\nonumber\\
&=\mathbb{E}[\tilde{f}(\bar{\bsx}_k)-nf^*+\tilde{f}(\bar{\bsx}_{k+1})
-\tilde{f}(\bar{\bsx}_k)]\nonumber\\
&\overset{\text{(m)}}{\le}\mathbb{E}[\tilde{f}(\bar{\bsx}_k)-nf^*
-\eta(\bar{\bsg}_{k}^e)^\top\bsg^0_k
+\frac{1}{2}\eta^2L_f\|\bar{\bsg}_{k}^e\|^2]\nonumber\\
&\overset{\text{(n)}}{=}W_{4,k}
-\eta(\bar{\bsg}_{k}^s)^\top\bar{\bsg}^0_k
+\frac{1}{2}\eta^2L_f\mathbb{E}[\|\bar{\bsg}_{k}^e\|^2]\nonumber\\
&\overset{\text{(o)}}{=}W_{4,k}
-\frac{1}{2}\eta(\bar{\bsg}_{k}^s)^\top(\bar{\bsg}^s_k+\bar{\bsg}^0_k-\bar{\bsg}^s_k)\nonumber\\
&~~~-\frac{1}{2}\eta(\bar{\bsg}^s_{k}-\bar{\bsg}^0_k+\bar{\bsg}^0_k)^\top\bar{\bsg}^0_k
+\frac{1}{2}\eta^2L_f\mathbb{E}[\|\bar{\bsg}_{k}^e\|^2]\nonumber\\
&\overset{\text{(p)}}{\le} W_{4,k}-\frac{1}{4}\eta(\|\bar{\bsg}^s_{k}\|^2
-\|\bar{\bsg}^0_k-\bar{\bsg}^s_k\|^2+\|\bar{\bsg}_{k}^0\|^2\nonumber\\
&~~~-\|\bar{\bsg}^0_k-\bar{\bsg}^s_k\|^2)
+\frac{1}{2}\eta^2L_f\mathbb{E}[\|\bar{\bsg}_{k}^e\|^2]\nonumber\\
&\overset{\text{(q)}}{\le} W_{4,k}-\frac{1}{4}\eta\|\bar{\bsg}^s_{k}\|^2
+\|\bsx_k\|^2_{\eta L_f^2\bsK}\nonumber\\
&~~~+\frac{np}{4}\eta L_f^2\delta^2_k-\frac{1}{4}\eta\|\bar{\bsg}_{k}^0\|^2
+\frac{1}{2}\eta^2L_f\mathbb{E}[\|\bar{\bsg}^e_{k}\|^2],\label{zerosg:v4k}
\end{align}
where (m) holds since that $\tilde{f}$ is smooth; (n) holds due to $\mathbb{E}[\bsg^e_k] = \bsg^s_k$, $x_{i,k}$ and $v_{i,k}$ are independent; (o) holds due to $(\bar{\bsg}_{k}^s)^\top\bsg^0_k=(\bsg_{k}^s)^\top\bsH\bsg^0_k=(\bsg_{k}^s)^\top\bsH\bsH\bsg^0_k
=(\bar{\bsg}_{k}^s)^\top\bar{\bsg}^0_k$; (p) holds due to the Cauchy--Schwarz inequality; and (q) holds due to  $\|\bsg^s_k-\bsg_k^0\|^2\le 2L_f^2\|\bsx_{k}\|^2_{\bsK}+\frac{np}{2}L_f^2\delta_k^2$.

(v) 
Define $W_{k+1} = \sum_{i = 1}^{4}W_{i, k+1}$ and then we have the following inequality holds.

\begin{align}
&\mathbb{E}[W_{k+1}]\nonumber\\
&\le W_{k}-\|\bsx_k\|^2_{\eta\alpha\bsL-\frac{1}{2}\eta\bsK
-\frac{3}{2}\eta^2\alpha^2\bsL^2-\eta(1+5\eta)L_f^2\bsK}\nonumber\\
&~~~+\Big\|\bm{v}_k+\frac{1}{\beta}\bsg_k^0\Big\|^2_{\frac{3}{2}\eta^2\beta^2\bsK}+nL_f^2\eta\Big[\frac{p}{4}+(\frac{p}{4}+4)\eta\Big]\delta_k^2\nonumber\\
&~~~+2\eta^2
\mathbb{E}[\|\sigma(\bsg^e_k, \gamma)\|^2]\nonumber\\
&~~~+\frac{1}{2}\eta\Big(\frac{1}{\rho_2(L)}+\kappa_1\Big)
\Big\|\bm{v}_k+\frac{1}{\beta}\bsg_{k}^0\Big\|^2_{\bsK}\nonumber\\
&~~~+\|\bsx_k\|^2_{\eta^2\beta^2(\bsL+\kappa_1\bsL^2)}\nonumber\\
&~~~+\frac{\eta}{\beta^2}\Big(\eta+\frac{1}{2}\Big)\Big(\frac{1}{\rho_2(L)}+\kappa_1\Big)L_f^2
\mathbb{E}[\|\bar{\bsg}^e_{k}\|^2]\nonumber\\
&~~~+\|\bm{x}_k\|^2_{\eta(\beta\bsL+\frac{1}{2}\bsK)
+\eta^2(\frac{1}{2}\alpha^2-\alpha\beta+\beta^2)\bsL^2
+\eta(1+3\eta)L_f^2\bsK}\nonumber\\
&~~~+\eta^2\Big[1 + (\frac{1}{2\eta\beta^2}+\frac{3}{2\beta^2}\Big)L_f^2\Big]\mathbb{E}[\|\sigma(\bsg^e_k, \gamma)\|^2]\nonumber\\
&~~~+nL_f^2\eta\Big[\frac{p}{4}+(\frac{p}{4}+2)\eta\Big]\delta^2_k\nonumber\\
&~~~-\Big\|\bm{v}_k+\frac{1}{\beta}\bsg_{k}^0\Big\|^2_{\eta(\beta-\frac{1}{2}
-\eta\beta^2)\bsK}-\frac{1}{4}\eta\|\bar{\bsg}^s_{k}\|^2\nonumber\\
&~~~+\|\bsx_k\|^2_{\eta L_f^2\bsK}\nonumber\\
&~~~+\frac{np}{4}\eta L_f^2\delta^2_k-\frac{1}{4}\eta\|\bar{\bsg}_{k}^0\|^2
+\frac{1}{2}\eta^2 L_f\mathbb{E}[\|\bar{\bsg}^e_{k}\|^2]\nonumber\\
&\overset{\text{(r)}}{\le} W_{k}-\|\bsx_k\|^2_{\eta\bsM_{1}-\eta^2\bsM_{2}-b_{1}\bsK}-\Big\|\bm{v}_k+\frac{1}{\beta}\bsg_{k}^0\Big\|^2_{b^0_{2}\bsK}\nonumber\\
&~~~-\eta\Big(\frac{1}{4}-6c_1(p-1)\Big)\|\bar{\bsg}^0_{k}\|^2\nonumber\\
&~~~+\underbrace{c_1\Big[ 6(p-1)\sigma^2_2+(\frac{3}{n_c}+2)p\sigma^2_1 \Big]}_{\mathcal{O}(np)\eta^2}+\underbrace{c_{3}\eta\delta_k^2L_f^2}_{\mathcal{O}(np^2)\eta\delta_k^2},
\label{zo_pb:zerosg:vkLya}
\end{align}
where (r) holds due to \eqref{zo_pb:zerosg:rand-grad-esti2}, \eqref{zo_pb:zerosg:rand-grad-esti4}, $\alpha=\kappa_1\beta$,  $\eta=\frac{\kappa_2}{\beta}$, and
\begin{align*}
\bsM_{1}&=(\alpha-\beta)\bsL-\Big(1+3L_f^2+\frac{6L_f^4\kappa_1}{\beta^2}(p-1)\Big)\bsK,\\
\bsM_{2}&=\beta^2\bsL+(2\alpha^2+\beta^2)\bsL^2+8L_f^2\bsK\\
&\quad+6(p-1)\Big(3+\frac{1}{2}L_f+\frac{2L_f^{2}}{\beta^2}\kappa_1 + \frac{L_f^2}{2\beta^2}\Big)\bsK,\\
\kappa_3&=\frac{1}{\rho_2(L)}+\kappa_1+1,\\
b^0_{2}&=\frac{1}{2}\eta(2\beta-\kappa_3)-2.5\kappa_2^2,\\
b_{1}
&=6p\kappa_3L_f^4\frac{\eta}{\beta^2}+12p(\kappa_3+1)L_f^4\frac{\eta^2}{\beta^2},\\
c_1&= \Big(3+\frac{1}{2}L_f+\frac{2L_f^{2}}{\beta^2}\kappa_1 + \frac{L_f^2}{2\beta^2}\Big)n\eta^2 + \frac{L_f^2\kappa_1}{\beta^2}n\eta,\\
c_2 &=\frac{3}{4}pn+\eta n(\frac{p}{2} +6) +\frac{p^2L_f^2}{2\beta^2}\kappa_1,\\
c_3 &=c_2 + \Big(c_1 - \frac{L_f^2\kappa_1}{\beta^2}n\eta\Big)p^2\eta.		
\end{align*}

Consider $p\geq 1$, $\alpha=\kappa_1\beta$, $\kappa_1>1$, $\beta$ is large enough, and $\eta=\frac{\kappa_2}{\beta}$, we have
\begin{align}
\eta\bsM_{1}&\ge [(\kappa_1 - 1)\rho_2(L) - 1]\kappa_2\bsK.\label{zo_pb:zerosg:m1-rand-pd}\\
\eta^2\bsM_{2}&\le [\rho(L)+(2\kappa_1^2+1)\rho(L^2)+1]\kappa_2^2\bsK.\label{zo_pb:zerosg:m2-rand-pd}\\
b^0_{2}&\ge\frac{1}{2}(\kappa_2 - 5\kappa_2^2).\label{zo_pb:zerosg:vkLya-b1}
\end{align}

From~\eqref{zo_pb:zerosg:vkLya}--~\eqref{zo_pb:zerosg:vkLya-b1}, let $\kappa_4 = [(\kappa_1 - 1)\rho_2(L) - 1]\kappa_2- [\rho(L)+(2\kappa_1^2+1)\rho(L^2)+1]\kappa_2^2 $ we know that~\eqref{zo_pb:zerosg:sgproof-vkLya2T} holds.

Similar to the way to get~\eqref{zo_pb:zerosg:sgproof-vkLya2T}, we have~\eqref{zo_pb:zerosg:v4kspeed}.

\end{proof}

\end{document}